\documentclass[12pt]{amsart}
\usepackage{epsfig,color}
\usepackage{blindtext}
\usepackage{hyperref}
\usepackage{graphicx}
\usepackage{enumitem} 
\usepackage{url}
\usepackage{amssymb}
\usepackage{graphicx,import}
\usepackage{comment}

\theoremstyle{definition}

\newtheorem{theorem}{Theorem}[section]
\newtheorem{definition}[theorem]{Definition}
\newtheorem{conjecture}[theorem]{Conjecture}
\newtheorem{proposition}[theorem]{Proposition}
\newtheorem{lemma}[theorem]{Lemma}
\newtheorem{remark}[theorem]{Remark}
\newtheorem{corollary}[theorem]{Corollary}

\newtheorem*{remark*}{Remark}

\numberwithin{equation}{section}

\newcommand{\mc}{\mathcal}
\newcommand{\mb}{\mathbb}

\newcommand{\eps}{\varepsilon}

\newcommand{\R}{\mathbb{R}}

\DeclareMathOperator{\vol}{Vol}

\title[Generic Multiplicity One Mean Curvature Flow]{Local Entropy and Generic Multiplicity One  Singularities of Mean Curvature Flow of Surfaces}

\date{\today}

\author{Ao Sun}
\address{Department of Mathematics, Massachusetts Institute of Technology, Cambridge, MA 02139, USA}
\address{Current address: Department of Mathematics,
	University of Chicago,
	5734 S. University Avenue,
	Chicago, IL 60637, USA}
\email{aosun@uchicago.edu}

\begin{document}

\begin{abstract}
In this paper we prove that the generic singularities of mean curvature flow of closed embedded surfaces in $\R^3$ modeled by closed self-shrinkers with multiplicity has multiplicity one. Together with the previous result by Colding-Minicozzi in \cite{colding2012generic}, we conclude that the only generic singularity of mean curvature flow of closed embedded surfaces in $\R^3$ modeled by closed self-shrinkers is a multiplicity one sphere. We also construct particular perturbations of the flow to avoid those singularities with multiplicity higher than one. Our result partially addresses the well-known multiplicity one conjecture by Ilmanen.
\end{abstract}

\maketitle

\section{Introduction}
Let $\{M_t\}_{t\in [0,T)}$ be a family of hypersurfaces in $\R^{n+1}$. We say they are flowing by mean curvature if they satisfy the equation
\[\partial_t x=-H\mathbf{n}.\]
Here $x$ is the position, $H$ is the mean curvature, and $\mathbf{n}$ is the normal vector field. Mean curvature flow is the negative gradient flow for the area functional, and has been studied in a number of ways in the past 40 years. It is well-known that any closed hypersurface in $\R^{n+1}$ must generate singularities in finite time, and a central topic in the study of mean curvature flow is understanding these singularities. 

Given a spacetime point $(y,s)\in\R^{n+1}\times[0,T]$, we define the blow up sequence of $M_t$ at $(y,s)$ to be a sequence of flows $\{M_t^{\alpha_i}\}:=\{\alpha_i(M_{s+\alpha_i^{-2}t}-y)\}$ with $\alpha_i\to \infty$, which means that we translate $M_{s+\alpha_i^{-2}t}$ by $-y$ and then rescale it by $\alpha_i$. Using Huisken's monotonicity formula (see \cite{huisken1990asymptotic}), Ilmanen \cite{ilmanen1995singularities} and White \cite{white94partial} proved that $\{M_t^{\alpha_i}\}$ weakly converges to a homothetic weak mean curvature flow (which is known to be the Brakke flow, cf. \cite{brakke2015motion}) $\{\sqrt{-t}\Sigma\}_{t\in(-\infty,0)}$. This limit is known to be the \emph{tangent flow} at the singular point. We call $\Sigma$ a self-shrinker, and $\Sigma$ weakly satisfies the elliptic equation
\[H=\frac{\langle x,\mathbf{n}\rangle}{2}.\]

Moreover, Ilmanen \cite{ilmanen1995singularities} proved that if $\{M_t\}_{t\in[0,T)}$ is a family of closed embedded surfaces in $\R^3$ flowing by mean curvature, and at time $T$ it generates a singularity, then the blow up limit is actually a smooth embedded self-shrinker in $\R^3$. However, the convergence itself may not be smooth, in other words the convergence may have multiplicity higher than one. Ilmanen made the following famous conjecture:

\begin{conjecture}{(Multiplicity One Conjecture, see \cite[p.7]{ilmanen1995singularities})}
For a family of closed smooth embedded surfaces in $\R^3$ flowing by mean curvature, every tangent flow at the first singular time has multiplicity one.
\end{conjecture}

As far as we know, this conjecture is still open. In this paper, we address a generic version of this conjecture. Roughly speaking, we proved that for a smooth family of closed embedded surfaces in $\R^3$ flowing by mean curvature, if one tangent flow at the first singular time is $m\sqrt{-t}\Sigma$ with multiplicity $m>1$, then we can perturb the flow such that $m\sqrt{-t}\Sigma$ can never be the tangent flow of the perturbed flow. More precisely,

\begin{theorem}\label{Thm:Main Theorem}
Suppose $\{M_t\}_{t\in [0,T)}$ is a family of smooth closed embedded surfaces in $\mb R^{3}$ flowing by mean curvature, and suppose a tangent flow of $M_t$ at $(y,T)$ for some $y\in\R^3$ is $\{\sqrt{-t}\Sigma\}_{t\in(-\infty,0)}$ with multiplicity $m>1$. Given $\epsilon_0>0$, there exists $t_0\in[0,T)$ such that there is a perturbed closed embedded surface $\tilde M_{t_0}$ which is isotopic to $M_{t_0}$, has Hausdorff distance to $M_{t_0}$ less than $\epsilon_0\sqrt{T-t_0}$, and if $\{\tilde M_t\}_{t\in[t_0,T')}$ is a family of closed embedded surfaces flowing by mean curvature starting from $\tilde M_{t_0}$, the tangent flow will never be $m\sqrt{-t}\Sigma$. \end{theorem}

The idea follows from the significant paper \cite{colding2012generic} by Colding-Minicozzi. In \cite{colding2012generic}, Colding-Minicozzi proved that among all the multiplicity one self-shrinkers in $\R^{n+1}$, the sphere $S^n$ and the generalized cylinders $S^k\times\R^{n-k}$ are the only generic singularities, in the sense that they cannot be perturbed away. In particular, the sphere $S^n$ is the only closed generic singularity among all multiplicity one singularities. Colding-Minicozzi used the fact that the sphere is the only closed entropy stable self-shrinker to prove the genericity. We have a similar result for self-shrinkers with higher multiplicity.

\begin{theorem}\label{Thm: higher multiplicity self-shrinkers are entropy unstable}
Let $\Sigma$ be a closed self-shrinker, and $m>1$ be a positive integer. Then $m\Sigma$ is entropy unstable. 
\end{theorem}

Our main theorem shows that in $\R^3$, among all the closed self-shrinkers with possible higher multiplicity, the multiplicity one self-shrinkers are the only generic singularities. Thus combining our result with the result by Colding-Minicozzi, we prove that the only closed generic singularity of mean curvature flow of closed embedded surfaces in $\R^3$ is the multiplicity one sphere.

\vspace{10pt}
\subsection*{Mean Curvature Flow and Singularities} Mean curvature flow was first studied in material science in the 1950s. Mullins may have been the first to write down the mean curvature flow, introducing the mean curvature flow equation when he was studying coarsening in metals. The modern study of mean curvature flow was initiated by Brakke. In his monograph \cite{brakke2015motion} which was first published in 1978, Brakke studied the weak mean curvature flow, which is known as the Brakke flow. Later Ilmanen \cite{ilmanen1994elliptic} generalized it to a much more general setup.

Since then, there are have been studies on mean curvature flow. Huisken first studied mean curvature flow from a differential geometry perspective, and introduced his famous monotonicity formula, see \cite{huisken1990asymptotic}. Osher-Sethian \cite{osher-sethian} used the level set flow method to study mean curvature flow numerically, and later this approach was rigorously developed by Chen-Giga-Goto \cite{chen-giga-goto} and Evans-Spruck \cite{evans-spruck}. We refer to the surveys \cite{white2002ICM}, \cite{colding-minicozzi-pedersen} for general discussion and background. 

A family of closed hypersurfaces flowing by mean curvature in $\R^{n+1}$ must generate singularities in finite time, and the singularities are modelled by self-shrinkers. Understanding the singularities is one of the central topics in the study of mean curvature flow. 

For closed curves flowing by mean curvature in the plane, Gage-Hamilton \cite{gage-hamilton} and Grayson \cite{grayson} proved that all closed embedded curves in the plane will shrink to a single point under the mean curvature flow. Moreover, the singularity is spherical, i.e. the tangent flow at the singularity is a circle. It is also the only closed embedded self-shrinker in one-dimensional case. Later Abresch-Langer \cite{abresch-langer} classified all immersed self-shrinkers in one-dimensional case, which are now called the Abresch-Langer curves.

The higher dimensional cases are much more complicated. Under a convexity assumption, Huisken \cite{huiksen1984convex} proved that the only possible singularity that arises is modelled by the sphere. With other curvature assumptions like mean convexity (cf. \cite{white94partial}, \cite{huisken-sinestrari1999}, \cite{huisken-sinestrari1999-2}, \cite{white2000meanconvex}, \cite{white2003thenature}, \cite{haslhofer-kleiner}, \cite{brendle-huisken}, \cite{haslhofer-kleiner-2}), we also know what possible singularity models can arise. The main reason is that we can classify the self-shrinkers with those curvature assumptions. For example, Huisken \cite{huisken1990asymptotic} and Colding-Minicozzi \cite{colding2012generic} classified all the mean convex self-shrinkers. Without curvature assumptions, we know very little about the classification of the self-shrinkers. For example, Brendle \cite{brendle-genus0} proved that we can classify the self-shrinkers in $\R^3$ with genus $0$, but with more complicated topology no classification result is known.

These curvature assumptions also imply that the tangent flow at the singularities must be multiplicity one self-shrinkers. The reason is that those curvature assumptions provide the so-called ``non-collapsing" property, cf. \cite{white2000meanconvex}, \cite{andrews}, \cite{haslhofer-kleiner}. Roughly speaking, ``non-collapsing" property implies that the blow up sequence cannot converge to the limit which is a self-shrinker in a singular manner, hence can not have high multiplicity.

Without these curvature assumptions, the problem is much more complicated. There are a number of constructions of self-shrinkers by Kapouleas-Kleene-Moller \cite{kapouleas-kleene-moller}, Moller \cite{moller2011closed}, Nyugen \cite{nguyen1}, \cite{nguyen2}, etc. The fruitful examples suggest that it is impossible to classify all self-shrinkers in higher dimensions. Also, when the dimension of the flow is larger than $2$, we do not even know that whether the tangent flow is a smooth self-shrinker. The non-collapsing property also hasn't been developed yet. 

\vspace{10pt}
\subsection*{Generic Singularity and Entropy} 
All the difficulties suggest that it would be too complicated studying all the mean curvature flow singularities. So Huisken and Angenent-Chopp-Ilmanen \cite{angenent-ilmanen-chopp} suggested to study the mean curvature flow starting at a generic closed embedded hypersurface. 

The idea of studying generic objects is widely used in geometric analysis. For example, White \cite{white-bumpy} studied a generic set of metrics, called bumpy metrics, and recently it is used by Irie-Marques-Neves \cite{irie-marques-neves} to prove Yau's conjecture of the existence of infinitely many minimal surfaces for manifolds with generic metrics. Another example is in gauge theory. In Donaldson's work on the topology of four-manifolds, in order to find a moduli space of the solutions of self-duality equations with isolated singularities, a perturbation of the space is necessary. Donaldson proved that a generic perturbation would give us a nice moduli space, see \cite{donaldson_4manifold}.

In \cite{colding2012generic}, Colding-Minicozzi proved that the sphere and the generalized cylinders are the only generic singularity models, in the sense that they are the only entropy stable self-shrinkers. The entropy of a hypersurface $\Sigma$ is defined to be
\[\lambda(\Sigma)=\sup_{(x_0,t_0)\in\R^{n+1}\times(0,\infty)}\frac{1}{(4\pi t_0)^{n/2}}\int_\Sigma e^{-\frac{|x-x_0|^2}{4t_0}}d\mu.\]
The entropy is translating invariant and scaling invariant. One can show that the self-shrinkers are critical points of entropy, and the entropy is achieved at $(x_0,t_0)=(0,1)$. Moreover, by Huisken's monotonicity formula, the entropy is non-increasing along with the mean curvature flow. Thus a self-shrinker $\Sigma$ with entropy larger than $M$ can never be the singularity model of mean curvature flow starting from $M$. 

Colding-Minicozzi defined a self-shrinker to be entropy unstable if there is a small perturbation of the self-shrinker which has strictly smaller entropy. Intuitively, this implies that we can perturb the flow to avoid this self-shrinker as the singularity model. In contrast, if the tangent flow is a multiplicity one sphere, then the flow becomes convex when approaching the singular spacetime point, hence it is still convex under small perturbation. By Huisken's result, the singularity model is still a sphere. Hence the sphere can not be perturbed away. Thus the sphere is entropy stable which also implies that it is dynamically stable. 

For immersed curve shortening flow in the plane, Baldauf and the author \cite{baldauf2018sharp} used this idea developed by Colding-Minicozzi to define a generic mean curvature flow of immersed curves in the plane. We showed that after small perturbations, a mean curvature flow of immersed curves only generates either a spherical singularity (hence it is a mean curvature flow of embedded curves), or a type II singularity, which is the singularity where curvature blows up super fast. The type II singularities are cusp singularities for the curve shortening flow.

Recently in a sequel of papers, Colding-Minicozzi proved that the entropy stability / instability implies the dynamic stability/instability respectively, see \cite{colding2018dynamics}, \cite{colding2018wandering}. More precisely, if $\Sigma$ is an entropy unstable self-shrinker with multiplicity one, then all but very few hypersurfaces near $\Sigma$ will leave a neighbourhood of $\Sigma$ under the mean curvature flow, modulo the translations and dilations. We hope our paper may serve as a starting point of this kind of analysis for singularities with multiplicity, and similar results should hold for higher multiplicity singularities, at least in $\R^3$.

\vspace{10pt}
\subsection*{Ideas of the Proof} 
Let us first recall some previous results on the multiplicity one problem. In \cite{li2016extension}, Haozhao Li and Bing Wang proved that a mean curvature flow with mean curvature decay must have the tangent flow to be a multiplicity one plane; in \cite{wang2016asymptotic}, Lu Wang proved that the tangent flow of the time $0$ slice of a complete self-shrinker is of multiplicity one. Both results work for surfaces flowing by mean curvature in $\R^3$. 

Because the problems they studied have certain curvature bounds, they proved the ``pseudo-locality" of the mean curvature flow, i.e. if the flow is multi-sheeted graphs at some time, then it must keep being multi-sheeted graphs in the future. Then parabolic maximum principle suggests that those graphs must be pushed away from each other. Thus higher multiplicity can not happen.

In our case, we do not have pseudo-locality, but thanks to Ilmanen \cite{ilmanen1995singularities}, before the first singular time there are many time slices of the mean curvature flow which are multi-sheeted layers away from a small neighbourhood of finitely many points. Each layer is homotopy to the limit self-shrinker by removing finitely many points. We can perturb these specific time slices to reduce the entropy. The perturbations are inherited from those perturbations on the limited self-shrinker $\Sigma$ with multiplicity $m$. 

In order to do this, we develop the notion of entropy instability for self-shrinkers with multiplicities, and then prove that every self-shrinker with multiplicity higher than $1$ must be entropy unstable, and find the perturbation which strictly reduces the entropy. Since at a certain time slice the flow (after parabolic rescaling) is multi-sheeted layers over the self-shrinker, we can apply the same variation on the highest layer to reduce the entropy. This perturbed surface can not generate a singularity modelled by $m\Sigma$. 

The main technical issue is that Colding-Minicozzi's entropy is not continuous with respect to the topology in the space of measures, such as the varifold topology. The time slices after parabolic rescaling are only multi-sheeted layers away from a small neighbourhood of finitely many points, and it is impossible to control the geometry inside those small neighbourhoods of finitely many points. Thus the entropy of those time slices may be achieved at those places with small scale, and then we do not know whether those perturbations would reduce the entropy.

In order to overcome this issue, we introduce the notion of local entropy. Given a compact set $I\subset(0,\infty)$, we define the local entropy
\[\lambda^{[a,b]}(\Sigma)=\sup_{(x_0,t_0)\in\R^{n+1}\times[a,b]}\frac{1}{(4\pi t_0)^{n/2}}\int_\Sigma e^{-\frac{|x-x_0|^2}{4t_0}}d\mu.\] 
The local entropy is also translation invariant and monotone along with the mean curvature flow. However it is not dilation invariant, hence it can not be achieved at very tiny scales, which means the local entropy can not be achieved at those places where the geometry is wild. As a result, local entropy decreases sharply under perturbation. Then the similar dynamic argument rules out the singularity modelled by $m\Sigma$ after perturbation in the long term future.

One more technical remark is that the local entropy argument can not rule out the singularity modelled by $m\Sigma$ immediately. In order to rule out this possibility, we find a small ball lying in the region bounded by those perturbed surfaces, and the parabolic maximum principle tells us that the perturbed surfaces can not shrink to a single point before the small ball shrink to a single point. Note that if there is a singularity modelled by a closed self-shrinker, then the flow must shrink to a single point at the singular time. Hence any singularity modelled by closed self-shrinkers, in particular $m\Sigma$, can not show up immediately. Thus we conclude our main theorem.

Some techniques developed in this paper may have an independent interest. The analysis of local entropy, the construction of the perturbations may be applied to other problems in mean curvature flow and other fields.

We also want to emphasize that we do not care much about the precise regularity of the flow. Thus in this paper we consider all the surfaces to be smooth surfaces, or the perturbed $C^{2,\alpha}$ surfaces. 

\vspace{10pt}
\subsection*{Organization of the Paper}\mbox{}

In Section \ref{S:Preliminary}, we introduce some basic notions for mean curvature flow. 

In Section \ref{S:Entropy and Local Entropy}, we define entropy and local entropy, and study their properties. One key property is the continuity of entropy/local entropy under certain circumstances.

In Section \ref{S:Entropy Instablity}, we introduce the entropy instability defined by Colding-Minicozzi for self-shrinkers and generalize it to self-shrinkers with multiplicity.

In Section \ref{S:Ilmanen's Analysis of MCF Near Singularity}, we sketch Ilmanen's analysis of blow up sequence near a singularity, in particular we illustrate how one can get the local multi-layered property of those time slices.

In Section \ref{S:Perturbation Near Singularities with Multiplicity}, we construct the desired perturbation over the time slices near the singularity.

In Section \ref{S:Proof of Main Theorem}, we prove our main Theorem \ref{Thm:Main Theorem}.

Finally we have an Appendix to collect some results we use.

\vspace{10pt}
\subsection*{Acknowledgement} I am grateful to Professor Bill Minicozzi for his helpful suggestions and comments. I am grateful to Zhichao Wang for helpful discussions. I want to thank Qiang Guang for reading an early version of this paper and his suggestions and comments. I am grateful to the anonymous referee's valuable suggestions and comments.

\section{Preliminary}\label{S:Preliminary}

\subsection{Geometric Measure Theory}\label{SS:Geometric Measure Theory}
We first recall some concepts in geometric measure theory. For more detailed discussion, see \cite{simon1983lecture} and \cite{ilmanen1994elliptic}.

We use $\mc H^k$ to denote the $k$-dimensional Hausdorff measure in $\mb R^n$. Let $X$ be an $\mc H^k$-measurable subset of $\mb R^n$, and let $\theta\colon\mb R^n\to \mb Z^+$ be a locally $\mc H^k$ integrable integer function, whose support is $X$. Then we can define a Radon measure $\mu(X,\theta)=\mc H^k\lfloor \theta$.

We say a Radon measure $\mu$ is \emph{integer $k$-rectifiable} if $\mu$ has $k$-dimensional tangent planes of positive multiplicity $\mu$-almost everywhere. Equivalently, this means that $\mu=\mu(X,\theta)$ for an $X$ which is $\mc H^k$-measurable and countably $k$-rectifiable, where $\theta$ is locally $\mc H^k$-integrable. 

Let $G^k\mb R^n$ be the total space of the Grassmann bundle of all $k$-planes of $\mb R^n$. A \emph{varifold} $V$ is a Radon measure on $G^k\mb R^n$. We can write
\[V(\psi)=\int\psi(x,S)dV(x,S),\]
where $\psi\in C_c^0(G^k\mb R^n,\mb R)$ and for each $x\in\mb R^n$, $S$ is the $k$-planes of $T_x\mb R^n$. Then for a $k$-rectifiable Radon measure $\mu$, we can associate an interger rectifiable $k$-varifold $V_\mu$ to $\mu$ defined by
\[V_\mu(\psi)=\int\psi(x,S)d V_\mu(x,S)=\int\psi(x,T_x\mu)d\mu(x)\]
for all test function $\psi \in C_c^0(G^k\mb R^n,\mb R)$. 

In this paper, we only consider the Radon measures and varifolds of smooth embedded hypersurfaces in $\mb R^{n+1}$, i.e. the varifolds associated with the summation of $\mu(\Sigma,m)$ where $\Sigma$ is a smooth embedded hypersurface in $\mb R^3$ and $m$ is a constant function on $\Sigma$. We will use $m\Sigma$ to denote $\mu(\Sigma,m)$ and the associated varifold $V$ in this case. If $V$ is associated with $\sum_{i=1}^k(\Sigma_i,m_i)$ the Radon measures of finitely many smooth embedded hypersurfaces, then we use $\sum_{i=1}^k m_i\Sigma_i$ to denote this Radon measure and the associated varifold $V$.

\subsection{$F$-Functional and First Variational Formula}

\begin{definition}
Let $\mu$ be an integer $n$-rectifiable Radon measure. Given $(x_0,t_0)\in\R^{n+1}\times(0,\infty)$, the $F_{x_0,t_0}$-functional of $\mu$ is defined to be the integral
\begin{equation}
F_{x_0,t_0}(\mu)=\frac{1}{(4\pi t_0)^{n/2}}\int e^{-\frac{|x-x_0|^2}{4t_0}} d\mu.
\end{equation}
\end{definition}

$F$-functional is a very important quantity in the study of mean curvature flow. Huisken proved the following monotonicity formulas for mean curvature flow.

\begin{theorem}\label{Thm:Huisken's monotonicity formula for F-functional}
{(Huisken's monotonicity formula, \cite[Theorem 3.1]{huisken1990asymptotic}.)} Let $\{M_t\}_{t\in[0,T)}$ be a family of hypersurfaces flowing by mean curvature, then for fixed $(x_0,t_0)\in\mathbb{R}^{n+1}\times (0,\infty)$ and $0\leq t_1\leq t_2<T$,
\[F_{x_0,t_0}(M_{t_2})\leq F_{x_0,t_0+t_2-t_1}(M_{t_1}).\]
If the equality holds for all $0\leq t_1\leq t_2<T$, then $M_t$ is a rescaling of a self-shrinker.
\end{theorem}

Later Ilmanen also generalized the monotonicity formulas to Brakke flows, see \cite[Lemma 7]{ilmanen1994elliptic}. 

In \cite{colding2012generic}, Colding-Minicozzi have studied the variations of the $F$-functional. They compute the following first variational formula.

\begin{lemma}{\cite[Lemma 3.1]{colding2012generic}}  
Let $\Sigma_s\subset\R^{n+1}$ be a variation of $\Sigma$ with variation vector field $f\mathbf{n}$. If $x_s,t_s$ are variations of $x_0,t_0$ with $x'_0=y,t'_0=h$ respectively, then $\frac{\partial}{\partial s}(F_{x_s,t_s}(\Sigma_s))$ is
\begin{equation}
\frac{1}{(4\pi t_0)^{n/2}}\int_\Sigma \left[f\left(H-\frac{\langle x-x_0,\mathbf{n}\rangle}{2t_0}\right) + h\left(\frac{|x-x_0|^2}{4t_0^2}-\frac{n}{2t_0}\right) + \frac{\langle x-x_0,y\rangle}{2t_0} \right]e^{-\frac{|x-x_0|^2}{4t_0}}d\mu.
\end{equation}
\end{lemma} 

It is straightforward to see that this variational formula also holds if we replace $\Sigma$ by $\mu=\sum_{i=1}^k m_i\Sigma_i$ which is a Radon measure of embedded hypersurfaces.

We will use a corollary of this lemma later. cf. \cite[proof of Lemma 7.7]{colding2012generic}.

\begin{corollary}\label{Cor:lambda is achieved in interior}
Suppose $\mu=\sum_{i=1}^k m_i\Sigma_i$ is a Radon measure of closed embedded hypersurfaces. Then given $t_0>0$, $F_{x_0,t_0}(\mu)$ achieves its maximum at $x_0$ which lies in the convex hull of the union of all $\Sigma_i$'s. 
\end{corollary}

\begin{proof}
First note that for fixed $t_0$, since all $\Sigma_i$'s are closed, $\lim_{|x_0|\to\infty}F_{x_0,t_0}(\mu)=0$. Thus $\sup_{x_0\in\R^{n+1}}F_{x_0,t_0}(\mu)$ is achieved at some $\tilde x_0\in \R^{n+1}$. Then by the first variational formula $\frac{\partial}{\partial s}(F_{x_s,t_0}(\mu))=0$, hence
\[\frac{1}{(4\pi t_0)^{n/2}}\int \frac{\langle x-\tilde x_0,y\rangle}{2t_0} e^{-\frac{|x-\tilde x_0|^2}{4t_0}}d\mu=0\]
holds for all $y\in\R^{n+1}$. Thus $\tilde x_0$ must lie in the convex hull of $\bigcup_{i=1}^k\Sigma_i$.
\end{proof}

\section{Entropy and Local Entropy}\label{S:Entropy and Local Entropy}
\begin{definition}
Suppose $\mu$ is an integer $n$-rectifiable Radon measure. The \emph{entropy} $\lambda$ of $\mu$ is defined to be 
\begin{equation}
\lambda(\mu)=\sup_{(x_0,t_0)\in\R^{n+1}\times(0,\infty)}F_{x_0,t_0}(\mu).
\end{equation}
\end{definition}
In particular, if $\mu$ is the Radon measure associated to an embedded hypersurface (with multiplicity one), then this definition coincides with the definition by Colding-Minicozzi in \cite{colding2012generic}. 

The following propositions are straightforward generalizations of results in \cite{colding2012generic}.

\begin{proposition}{\cite[Lemma 7.2 \& Section 7.2]{colding2012generic}}
Suppose $\mu=\sum_{i=1}^k m_i \lambda(\Sigma_i)$ where $m_i$'s are positive integers and $\Sigma_i$'s are closed embedded hypersurfaces in $\mb R^{n+1}$, then the following properties hold:
\begin{enumerate}
\item $F_{x_0,t_0}(\mu)$ is a smooth function of $x_0,t_0$ on $\mb R^{n+1}\times(0,\infty)$;
\item $\lambda(\mu)<\infty$;
\item If $\Sigma_i$'s are all self-shrinkers, then $\lambda(\mu)=F_{0,1}(\mu)=\sum_{i=1}^k m_i\lambda(\Sigma_i)$.
\end{enumerate}
\end{proposition}

The entropy can be viewed as the supremum of the Gaussian functional $F_{0,1}$ of a Radon measure under all possible translations and dilations. So the entropy detects the global information of a measure from every scale. There is another quantity called \emph{volume growth} also detecting the global information of a measure from every scale. Let $\mu$ be a $n$-rectifiable Radon measure, its volume growth (if $\mu$ is $2$-rectifiable then we say ``area growth") is defined to be 
\[\sup_{R>0,x\in\mb R^{n+1}}\mu(B_R(x))/R^n.\]

The entropy and the volume growth are equivalent quantitatively.

\begin{proposition}\label{Prop: Entropy and area growth are equivalent}
There exists a constant $C>0$ only depending on $n$, such that for any $n$-rectifiable Radon measure $\mu$ defined on $R^{n+1}$,
\[C^{-1}\lambda(\mu)\leq \sup_{R>0,x\in\mb R^{n+1}}\mu(B_R(x))/R^n\leq C\lambda(\mu).\]
\end{proposition}

\begin{proof}
The proof is the same as the proof of \cite[Theorem 2.2]{sun2018singularities}. On one hand, for any $y\in\mb R^{n+1}$ and $R>0$, we have
\[\lambda(\Sigma)\geq \frac{1}{(4\pi t)^{n/2}}\int e^{\frac{-|x-y|^2}{4t}}\chi_{B_R(y)}d\mu \geq \frac{1}{(4\pi t)^{n/2}}e^{-R^2/4t}\mu(B_R(y)).\]
Then choose $t=R^2$ we have 
\[\frac{\mu(B_R(y))}{R^n}\leq C\lambda(\Sigma),\]
where $C$ is a universal constant.

On the other hand, since the statement is translation and scaling invariant, we only need to prove $\int e^{-|x|^2}d\mu\leq C \sup_{x\in\mb R^{n+1}}\sup_{R>0}\frac{\mu(B_R(x))}{R^n}$ to conclude the second statement. 
\begin{equation}
\begin{split}
\int e^{-|x|^2}d\mu&\leq\sum_{y\in\mb Z^{n+1}}\int e^{-|x|^2}\chi_{B_2(y)}d\mu\\
&\leq C\sum_{y\in\mb Z^{n+1}} e^{-|y|^2}\mu(B_2(y))\leq C\sum_{y\in\mb Z^{n+1}}e^{-|y|^2}\mu(B_2(y))\\
&\leq C \sup_{x\in\mb R^{n+1}}\sup_{R>0}\frac{\mu(B_R(x))}{R^n}.
\end{split}
\end{equation}
Here $C$ varies from line to line, but it is always a universal constant. Then we conclude this proposition.
\end{proof}

\vspace{5pt}
Now we are going to generalize the entropy to local entropy. 

\begin{definition}
Suppose $\mu$ is an integer $n$-rectifiable Radon measure. Given a subset $U\subset \mb R^{n+1}$ and $I\subset (0,\infty)$, we define the \emph{local entropy} $\lambda^I_U$ to be
\begin{equation}
\lambda^I_U(\mu)=\sup_{(x_0,t_0)\in U\times I}F(x_0,t_0)(\mu).
\end{equation}
If $U=\mb R^{n+1}$, we simplify the notation to be $\lambda^I(\mu)$.
\end{definition}

We have the following formulas for the local entropy.

\begin{proposition}\label{Prop:entropy sum and sum of entropy}
\[\lambda^I_U(m\Sigma)=m\lambda^I_U(\Sigma),\]
\[\lambda^I_U(\sum_{i=1}^k m_i\Sigma_i)\leq \sum_{i=1}^k m_i \lambda^I_U(\Sigma_i),\]
and the equality of the second formula holds if and only if all the $\Sigma_i$'s entropy are achieved at the same point $(x_0,t_0)$. 
\end{proposition}

\begin{proof}
The first identity is straightforward. For the second inequality, note that
\[\lambda(\sum_{i=1}^k m_i \Sigma_i)= \sup  F_{x_0,t_0}(\sum_{i=1}^k m_i \Sigma_i)\leq\sum_{i=1}^k m_i \sup F_{x_0,t_0}(\Sigma_i)\leq  \sum_{i=1}^k m_i \lambda(\Sigma_i).\]
When the equality hold is straightforward to see.
\end{proof}

Later we will study the blow up sequence of a mean curvature flow. So we also need to study how the local entropy changes after a parabolic rescaling. Recall that a parabolic rescaling of the flow $M_t$ at $(y,s)\in\mathbb{R}^{n+1}\times[0,T]$ of scale $\alpha$ is defined to be
\begin{equation}
M_{t}^{\alpha,(y,s)}\colon=\alpha(M_{s+\alpha^{-2} t}-y),
\end{equation}
which means that we first translate $M_{s+\alpha^{-2} t}$ by $y$ and then rescale it by $\alpha$ in $\mb R^{n+1}$. We will omit $(y,s)$ in the notation if the spacetime point is clear. 

Entropy is dilation invariant, hence it does not change under the parabolic rescaling. However local entropy is not dilation invariant. Let us first study how dilation changes the local entropy. Direct computation gives that

\[F_{x_0,t_0}(\alpha(\Sigma-y))=F_{\frac{x_0}{\alpha}+y,\frac{t_0}{\alpha^2}}(\Sigma),\]

So taking the supremum of $x_0\in\mb R^{n+1}$ gives

\begin{equation}\label{E:local entropy change under rescaling}
\lambda^{t_0}(\alpha(\Sigma-y))=\lambda^{\alpha^{-2}t_0}(\Sigma).
\end{equation}

Applying this identity to the parabolic rescaling of mean curvature flow gives

\begin{equation}\label{E:local entropy change under rescaling for MCF}
\lambda^{t_0}(M_{t}^{\alpha})=\lambda^{\alpha^{-2}t_0}(M_{s+\alpha^{-2} t}).
\end{equation}

\subsection{Continuity of Local Entropy}
In general the entropy is not continuous for Radon measures of embedded hypersurfaces. There are two reasons. First, the entropy is the supremum among all possible translations, so for a sequence of Radon measures of hypersurfaces, the entropy may be achieved at $x_i$ for $|x_i|\to \infty$, but the limit measure has entropy in a bounded region, then entropy cannot be continuous; second, the entropy is the supremum among all possible dilations, so for a sequence of Radon measures of hypersurfaces, the entropy may be achieved at $t_i$ for $t_i\to 0$, but the limit measure may have entropy in a bounded scale, then the entropy cannot be continuous.

However, we can show that the local entropy is continuous for a very wide family of Radon measures. We hope these continuity properties may serve as tools to study other problems which are related to mean curvature flow.

\begin{lemma}\label{L:F-functional is continuous}
Let $\mu_i$ be a sequence of $n$-rectifiable Radon measures and $\mu$ be a $n$-rectifiable Radon measure, with uniformly bounded entropy. Suppose $\mu_i\to \mu$ as Radon measure. Then given $(x_0,t_0)\in\mb R^n\times(0,\infty)$, $F_{x_0,t_0}(\mu_i)\to F_{x_0,t_0}(\mu)$.
\end{lemma}

\begin{proof}
We denote $\mu$ by $\mu_0$. $\mu_i\to \mu$ implies that for any $R>0$, 
\[\frac{1}{(4\pi t_0)^{n/2}}\int_{B_R(x_0)} e^{-\frac{|x-x_0|^2}{4t_0}} d\mu_i\to \frac{1}{(4\pi t_0)^{n/2}}\int_{B_R(x_0)} e^{-\frac{|x-x_0|^2}{4t_0}} d\mu.\]
So we only need to prove that 
\[\frac{1}{(4\pi t_0)^{n/2}}\int_{\mb R^{n+1}\backslash B_R(x_0)} e^{-\frac{|x-x_0|^2}{4t_0}} d\mu_i\to 0\]
uniformly as $R\to \infty$. This follows from the computation
\begin{equation}\label{eq:entropyupperbound}
\begin{split}
\frac{1}{(4\pi t_0)^{n/2}}\int_{\mb R^{n+1}\backslash B_R(x_0)} e^{-\frac{|x-x_0|^2}{4t_0}} d\mu_i&=\frac{1}{(4\pi t_0)^{n/2}}\sum_{j=1}^\infty \int_{\mb B_{R^{j+1}}(x_0)\backslash B_{R^j}(x_0)} e^{-\frac{|x-x_0|^2}{4t_0}} d\mu_i\\
&\leq \frac{1}{(4\pi t_0)^{n/2}}\sum_{j=1}^\infty \int_{\mb B_{R^{j+1}}(x_0)\backslash B_{R^j}(x_0)} e^{\frac{-R^{2j}}{4t_0}} d\mu_i\\
&\leq \frac{1}{(4\pi t_0)^{n/2}} \sum_{j=1}^\infty CR^{2j+2}e^{\frac{-R^{2j}}{4t_0}}
\end{split}
\end{equation}
where $C$ varies line by line. The third inequality comes from the fact that all $\mu_i$ has uniform entropy bound, hence uniform area growth (see. Proposition \ref{Prop: Entropy and area growth are equivalent}). So the last term in \eqref{eq:entropyupperbound} converges to $0$ as $R\to \infty$. Then we conclude the proof.
\end{proof}

Note that it is straightforward to see that the above Lemma holds for Radon measures of closed embedded hypersurfaces all lying in a fixed ball $B_R$, even without uniform entropy bound. From now on we will only consider the case that $\mu_i$ and $\mu$ are Radon measures of closed hypersurfaces lying in a ball $B_R$. Then the uniform entropy bound condition is not necessary.

\begin{lemma}\label{L:local entropy is continuous compact interval}
Let $R>0$ and $I$ be a compact subset of $(0,\infty)$. Let $\mu_i$, $\mu$ be $n$-rectifiable Radon measures of finitely many closed embedded hypersurfaces inside $B_R$. Suppose $\mu_i\to \mu$ as Radon measure, then $\lambda^I(\mu_i)\to \lambda^I(\mu)$.
\end{lemma}

\begin{proof}
Recall the corollary of the first variational formula (see Corollary \ref{Cor:lambda is achieved in interior}) gives that, for a fixed $t_0\in(0,\infty)$, the supremum of $F_{x_0,t_0}(\mu)$ for $x_0\in \mb R^{n+1}$ is achieved in the convex hull of $\mu$, i.e. inside the ball $B_R$. Hence $\lambda^{t_0}(\mu)$ is actually a maximum over a compact subset of $\mathbb{R}^{n+1}\times(0,\infty)$. Thus by the continuity of $F_{x_0,t_0}$ functional we get that
\[\lambda^I(\mu_i)=\sup_{(x_0,t_0)\subset B_R\times I}F_{x_0,t_0}(\mu_i)\to \sup_{(x_0,t_0)\subset B_R\times I}F_{x_0,t_0}(\mu)=\lambda^I(\mu).\]
\end{proof}

The above Lemma suggests that if we only need to take the supremum over a compact scale of $t_0$, then the local entropy is continuous. Next Lemma shows that the absence of a very tiny scale is enough for the continuity of the local entropy.

\begin{lemma}\label{L:local entropy is continuous open interval}
Let $R>0$ and $a\in(0,1)$. Let $\mu_i$, $\mu$ be $n$-rectifiable Radon measures of finitely many closed embedded hypersurfaces inside $B_R$. Suppose $\mu_i\to \mu$ as Radon measure, then 
\[\lambda^{[a,\infty)}(\mu_i)\to \lambda^{[a,\infty)}(\mu).\]
\end{lemma}

\begin{proof}
Denote $\mu$ by $\mu_0$ to simplify the notation. We only need to prove that there exists $b>0$ and $j>0$ such that 
\[\lambda^{[a,\infty)}(\mu_i)=\lambda^{[a,b]}(\mu_i)\ \ \ \text{for $i=0$ and $i\geq j$}.\]
Then the lemma is an application of Lemma \ref{L:local entropy is continuous compact interval}.

Since $F_{0,1}(\mu_i)\to F_{0,1}(\mu_0)$, there exists $j>0$ such that for $i>j$ we know that $F_{0,1}(\mu_i)\geq F_{0,1}(\mu_0)/2$. By (\ref{E:local entropy change under rescaling}), for $0< \alpha <1$ we have $\lambda^{t\alpha^{-2}}(\mu_i)=\lambda^t(\mu_i^\alpha)$. Note
\[\lambda^{t}(\mu_i^\alpha)=\sup_{x_0\in\mb R^{n+1}}F_{x_0,t}(\mu_i^\alpha)= \sup_{x_0\in\mb R^{n+1}}\frac{1}{4\pi t}\int e^{-\frac{|x-x_0|^2}{4t}}d\mu^\alpha_i\leq \frac{C}{t}\mu^\alpha_i(B_\alpha)\leq \frac{C}{t}\alpha^n.\]
Here in the last inequality we note that $\mu_i^\alpha$ is supported on $B_{\alpha R}$, and $\mu_i\to \mu$ implies the uniform volume bound. Thus, if we choose $\alpha$ small enough such that $C\alpha^n< F_{0,1}(\mu_0)/2$, then for $t\geq1$, 
\[\lambda^{t\alpha^{-2}}(\mu_i)< F_{0,1}(\mu_0)/2\leq F_{0,1}(\mu_i)\leq \lambda^{1}(\mu_i).\]
Then we get
\[\lambda^{[a,\infty)}(\mu_i)=\lambda^{[a,\alpha^{-2}]}(\mu_i) \text{ for $i=0$ and $i\geq j$}.\]
\end{proof}

As we mentioned before, in general the continuity of entropy does not hold even for the closed hypersurfaces. This is because the convergence of Radon measure may have a very complicated part on a small scale, which contributes to large entropy but vanishes in the limit. Thus in general the entropy is only lower semi-continuous. However, we can still prove the continuity of entropy under a certain condition that those complicated small scales can be ruled out.

Let us first introduce a notation. Given a closed hypersurface $\Sigma\subset\mathbb{R}^{n+1}$ and let $f$ be a smooth function on $\Sigma$, we define a new hypersurface $\Sigma_f$ to be
\[\Sigma_f\colon=\{x\colon x\in\Sigma+f\mathbf{n}\}.\]

\begin{lemma}\label{L:entropy is continuous for perturbation}
Suppose $\Sigma$ is a closed hypersurface in $\mathbb{R}^{n+1}$, and $f$ is a smooth function on $\Sigma$. Let $m$ be a positive integer. Then 
\[\lambda((m-1)\Sigma+\Sigma_{\eps f})\to\lambda(m\Sigma),\ \ \ \text{as $\eps\to 0$.}\]
\end{lemma}

\begin{proof}
Again, we only need to prove that there exists a compact set $U\times[a,b]\subset\mathbb{R}^{n+1}\times(0,\infty)$ such that when $\eps$ small enough,
\[\lambda((m-1)\Sigma+\Sigma_{\eps f})=\sup_{(x_0,t_0)\in U\times[a,b]}F_{x_0,t_0}((m-1)\Sigma+\Sigma_{\eps f}).\]

\vspace{5pt}
{\bf Step 1: $\lambda$ is achieved at bounded $x_0$.} Because $\Sigma$ is closed, there is $R>0$ such that $\Sigma\subset B_{R/2}$, and then for $\eps$ small enough $\Sigma_{\eps f}\subset B_R$. By the Corollary \ref{Cor:lambda is achieved in interior} of the first variational formula,
\[\lambda((m-1)\Sigma+\Sigma_{\eps f})=\sup_{(x_0,t_0)\subset B_R\times (0,\infty)}F_{x_0,t_0}((m-1)\Sigma+\Sigma_{\eps f}).\]

\vspace{5pt}
{\bf Step 2: $\lambda$ is achieved at $t_0$ with upper bound.} Let $V$ be the volume of $\Sigma$. When $\epsilon$ is small enough, $\Sigma_{\eps f}$ has volume no more than $2V$. Thus the volume of $(m-1)\Sigma+\Sigma_{\eps f}$ is no more than $(m+1)V$. Then, when $t_0^{n/2}>(m+1)V/F_{0,1}(\Sigma)$ we have

\begin{equation}
\begin{split}
F_{x_0,t_0}((m-1)\Sigma+\Sigma_{\eps f})&=\frac{1}{(4\pi t_0)^{n/2}}\int_{(m-1)\Sigma+\Sigma_{\eps f}}e^{-\frac{|x-x_0|^2}{4 t_0}}\\
&\leq \frac{1}{(4\pi t_0)^{n/2}}\vol((m-1)\Sigma+\Sigma_{\eps f})\\
&\leq \frac{1}{(4\pi t_0)^{n/2}}(m+1)V\leq F_{0,1}(\Sigma).
\end{split}
\end{equation}

Note $F_{0,1}((m-1)\Sigma+\Sigma_{\eps f})\to F_{0,1}(m\Sigma)$, so when $\epsilon$ small enough, let $b=(m+1)V/F_{0,1}(\Sigma)$ we can see that $\lambda ((m-1)\Sigma+\Sigma_{\eps f})$ cannot be achieved at $t_0>b$. Hence
\[\lambda((m-1)\Sigma+\Sigma_{\eps f})=\sup_{(x_0,t_0)\subset B_R\times (0,b]}F_{x_0,t_0}((m-1)\Sigma+\Sigma_{\eps f}).\]

\vspace{5pt}
{\bf Step 3: $\lambda$ is achieved at $t_0$ with lower bound.} Suppose $\lambda(m\Sigma)$ is achieved at $(y,s)$. It is known that the entropy of closed hypersurfaces are bounded below by the entropy of a sphere, see Bernstein-Wang \cite{Bernstein-Wang-2016}, Jonathan Zhu \cite{zhu2016entropy}, Ketover-Zhou \cite{ketover-zhou}. Hence
\[\lambda(m\Sigma)=m\lambda(\Sigma)\geq m\lambda(S^n)>m.\]
Suppose $\lambda(m\Sigma)=m+\delta$ for $\delta>0$. Since $F_{y,s}((m-1)\Sigma+\Sigma_{\eps f})\to F_{y,s}(m\Sigma)$, when $\epsilon$ is small enough $F_{y,s}((m-1)\Sigma+\Sigma_{\eps f})$ is greater than $m+\delta/2$, hence $\lambda((m-1)\Sigma+\Sigma_{\eps f})>m+\delta/2$ when $\epsilon$ is small enough.

Now we show that there exists $a>0$ such that when $t_0<a$, $F_{x_0,t_0}((m-1)\Sigma+\Sigma_{\eps f})<\lambda((m-1)\Sigma+\Sigma_{\eps f})$ for all $x_0\in\mathbb{R}^{n+1}$. Let us argue by contradiction. Suppose such $a$ does not exists, then there exists a sequence $\{\eps_i\}_{i=1}^\infty$, $\eps_i\to 0$ as $i\to \infty$, such that $\lambda((m-1)\Sigma+\Sigma_{\eps_i f})=F_{x_i,t_i}((m-1)\Sigma+\Sigma_{\eps_i f})$ where $t_i\to 0$ as $i \to \infty$. 

Denote $(m-1)\Sigma+\Sigma_{\eps f}$ by $\mu_i$. Let us estimate $F_{x_i,t_i}(\mu_i)$. Let $R_0>0$ to be determined. We can write $F_{x_i,t_i}(\mu_i)$ into two terms:
\begin{equation}
\begin{split}
F_{x_i,t_i}(\mu_i)&=\frac{1}{(4\pi t_i)^{n/2}}\int e^{-\frac{|x-x_i|^2}{4t_i}}d\mu_i\\
&=\frac{1}{(4\pi t_i)^{n/2}}\int_{|x-x_i|\geq R_0} e^{-\frac{|x-x_i|^2}{4t_i}}d\mu_i+\frac{1}{(4\pi t_i)^{n/2}}\int_{|x-x_i|<R_0} e^{-\frac{|x-x_i|^2}{4 t_i}}d\mu_i.
\end{split}
\end{equation}

The first term is bounded by
\[\frac{1}{(4\pi t_i)^{n/2}}e^{-\frac{R_0^2}{4 t_i}}\vol((m-1)\Sigma+\Sigma_{\eps f})\leq \frac{1}{(4\pi t_i)^{n/2}}e^{-\frac{R_0^2}{4t_i}}(m+1)V\leq C R_0^{-n},\]
where $C$ is a uniform constant, and in the last inequality we use the fact that $e^{x}\geq C x^{n/2}$. We choose $R_0$ such that $CR_0^{-n}\leq \delta/6$

For the second term, note that when $\epsilon$ small enough, $(m-1)\Sigma+\Sigma_{\eps f}$ has uniform curvature bound $|A|^2\leq C$. We translate $\mu_i$ by $x_0$ and then rescale $\mu_i$ by $t_i$ to get a new Radon measure of closed embedded hypersurfaces $\tilde\mu_i$, and the second term equals
\[\frac{1}{(4\pi t_i)^{n/2}}\int_{|x-x_i|<R_0} e^{-\frac{|x-x_i|^2}{4 t_i}}d\mu_i=\frac{1}{(4\pi )^{n/2}}\int_{|x|<t_i^{-1}R_0} e^{-\frac{|x|^2}{4}}d\tilde\mu_i.\]
Note that $\tilde\mu_i$ consists of hypersurfaces with curvature bound $|A|^2\leq C t_i^2$. So if we fix $R_i$ and let $i\to \infty$, this term will converge to the integral of Gaussian measure of hyperplanes through the origin, i.e. the multiplicity of $\mu_i$ at point $x_i$, $m$. So when $i$ large enough, this term is bounded by $m+\delta/6$. 

Combining the two terms we have
\[F_{x_i,t_i}((m-1)\Sigma+\Sigma_{\eps_i f})\leq m+\delta/3<m+\delta/2\] 
when $i$ large enough, which is a contradiction. Hence there exists $a>0$ such that when $t_0<a$, $F_{x_0,t_0}((m-1)\Sigma+\Sigma_{\eps f})<\lambda((m-1)\Sigma+\Sigma_{\eps f})$ for all $x_0\in\mathbb{R}^{n+1}$. 

Then we finish the proof.
\end{proof}

\subsection{Dynamics of Local Entropy}
Now we study the behavior of $F$-functional and local entropy under the mean curvature flow. Let $\{M_t\}_{t\in[0,T)}$ be a family of hypersurfaces flowing by mean curvature. Recall Huisken's monotonicity formula Theorem \ref{Thm:Huisken's monotonicity formula for F-functional}, we have
\[F_{x_0,t_0+{t_1-t_2}}(M_{t_2})\leq F_{x_0,t_0}(M_{t_1})\] 
for all $t_1\leq t_2$, $t_0\geq t_2-t_1$. Taking the supremum among $x_0\in\mathbb{R}^{n+1}$ and $t_0\in[a,b]$ for an interval $[a,b]\subset(0,\infty)$ gives
\begin{equation}\label{E:local entropy change under monotonicity}
\lambda^{[a+{t_1-t_2},b+{t_1-t_2}]}(M_{t_2})\leq \lambda^{[a,b]}(M_{t_1}).
\end{equation}

From the above computations, we know that the local entropy is no longer dilation invariant for a given hypersurface, hence no longer monotone decreasing under the rescaled mean curvature flow. However, it still can be the obstruction of singularities just like entropy. We now state the main theorem of this section.

\begin{theorem}\label{Thm:no certain type singularity in finite interval}
Given $0<a<1<b$. Let $M$ be a closed hypersurface, with 
\[\lambda^{[a\alpha^{-2},b\alpha^{-2}]}(M)<\lambda(m\Sigma),\]
where $m\Sigma$ is the Radon measures of a closed self-shrinker with multiplicity $m$. Let $\{M_t\}$ be the mean curvature flow starting from $M$ at time $T-\alpha^{-2}$. Then $m\Sigma$ can not be the blow up limit of $M_t$ in the time interval $[T+(a-1)\alpha^{-2},T+(b-1)\alpha^{-2}]$.
\end{theorem}

\begin{proof}
We argue by contradiction. Suppose $T_0\in [T+(a-1)\alpha^{-2},T+(b-1)\alpha^{-2}]$ is a singular time, and a blow up sequence of $M_t$ converge to $m\Sigma$, i.e. there exists $\alpha_i\to \infty$ such that $M_{t}^{\alpha_i,(y,T_0)}\to \sqrt{-t}m\Sigma$ as Radon measure. Then the sequence $M_i=M^{\alpha_i,(y,T_0)}_{-1}$ converge to $m\Sigma$ as Radon measure. Thus by the continuity of local entropy, we have
\[\lambda^{1}(M_i)\to \lambda^1(m\Sigma)=\lambda(m\Sigma).\]
By (\ref{E:local entropy change under rescaling}) and (\ref{E:local entropy change under monotonicity}), assuming that $\alpha_i$'s are large enough such that $(T_0-T+\alpha^{-2}-\alpha_i^{-2})>0$ we have
\begin{equation}
\begin{split}
\lambda^{1}(M_i)&=\lambda^{\alpha_i^{-2}}(M_{T_0-\alpha_i^{-2}})=\lambda^{\alpha_i^{-2}}(M_{T-\alpha^{-2}+(T_0-T+\alpha^{-2}-\alpha_i^{-2})})\\
&\leq \lambda^{\alpha_i^{-2}+(T_0-T+\alpha^{-2}-\alpha_i^{-2})}(M_{T-\alpha^{-2}})=\lambda^{T_0-T+\alpha^{-2}}(M_{T-\alpha^{-2}})\\
&\leq \lambda^{[a\alpha^{-2},b\alpha^{-2}]}(M).
\end{split}
\end{equation}
The last inequality comes from the assumption that $T_0\in [T+(a-1)\alpha^{-2},T+(b-1)\alpha^{-2}]$. However, $\lambda^{[a\alpha^{-2},b\alpha^{-2}]}(M)<\lambda(m\Sigma)$, hence $\lambda^{1}(M_i)$ can not converge to $\lambda(m\Sigma)=\lambda^1(m\Sigma)$, which is a contradiction.
\end{proof}

Note the only place in the proof we use the fact that $[a,b]\subset(0,\infty)$ is the continuity of local entropy. Hence by the continuity of local entropy even $b=\infty$ (Lemma \ref{L:local entropy is continuous open interval}), the proof works well for $b=\infty$. Thus we can generalize it into the following theorem:

\begin{theorem}\label{Thm:no certain type singularity in infinite interval}
Given $0<a<1$. Let $M$ be a closed hypersurface, with 
\[\lambda^{[a\alpha^{-2},\infty)}(M)<\lambda(m\Sigma),\]
where $m\Sigma$ is the Radon measures of a closed self-shrinker with multiplicity $m$. Let $\{M_t\}$ be the mean curvature flow starting from $M$ at time $T-\alpha^2$. Then $m\Sigma$ can not be the blow up limit of $M_t$ after time $T+(a-1)\alpha^{-2}$.

\end{theorem}

\begin{proof}
The proof is the same as the proof of Theorem \ref{Thm:no certain type singularity in finite interval}.
\end{proof}

\section{Entropy Instablity}\label{S:Entropy Instablity}
In \cite{colding2012generic}, Colding-Minicozzi studied the entropy stability of self-shrinkers. cf. \cite{liu2016index}. In this section, we review some of their ideas and generalize some of them to the Radon measures of embedded self-shrinkers.

Given a closed embedded self-shrinker $\Sigma\subset\mathbb{R}^{n+1}$, we say it is \emph{entropy unstable} if for given $\epsilon_0>0$, there exists a $C^{2,\alpha}$ function $f$ defined on $\Sigma$ such that $\|f\|_{C^{2,\alpha}}<\epsilon_0$ and
\[\lambda(\Sigma_{f})<\lambda(\Sigma).\]
 Otherwise we say it is entropy stable.

Entropy is defined for all the scales of a self-shrinker, hence it seems hard to characterize whether a self-shrinker is entropy stable or not. However, in \cite{colding2012generic}, Colding-Minicozzi proved that the entropy stability of self-shrinkers is equivalent to the stability of self-shrinkers as minimal surfaces in the Gaussian space. As a result, they can prove that the hyperplanes, spheres and generalized cylinders are the only entropy stable self-shrinkers. Moreover, for an unstable self-shrinker $\Sigma$, Colding-Minicozzi can construct a nearby hypersurface which has entropy strictly less than the entropy of $\Sigma$.

We want to generalize their analysis to self-shrinkers with multiplicity. 

\begin{definition}
Let $\Sigma$ be a closed embedded self-shrinker, and let $\mu=m\Sigma$. We say $\mu$ is \emph{entropy unstable} if for given $\epsilon_0>0$ there exist $C^{2,\alpha}$ functions $f_1\leq \cdots\leq f_m$ on $\Sigma$ such that $\|f_i\|_{C^{2,\alpha}}<\epsilon_0$ and 
\[\lambda(\sum_{i=1}^m\Sigma_{f_i})<\lambda(m\Sigma).\]
\end{definition}
Here $f_1\leq \cdots\leq f_m$ is a technical requirement for the study of embedded hypersurfaces. Note that this definition coincides with the definition by Colding-Minicozzi in \cite{colding2012generic} if the multiplicity $m=1$.

One essential fact is that all the closed self-shrinkers with multiplicity higher than $1$ are entropy unstable. This is our Theorem \ref{Thm: higher multiplicity self-shrinkers are entropy unstable}. We can construct an explicit perturbation, and Theorem \ref{Thm: higher multiplicity self-shrinkers are entropy unstable} is a direct corollary of this construction. This is the following theorem.

\begin{theorem}\label{Thm: Construction of unstable perturbation of self-shrinker with multiplicity}
Let $\Sigma$ be a closed self-shrinker, $m>1$ be a positive integer. There exists a smooth positive function $f$ defined on $\Sigma$ and $\epsilon_0>0$ such that for $\epsilon\in(0,\epsilon_0)$, 
\[\lambda((m-1)\Sigma+\Sigma_{\epsilon f})<\lambda(m\Sigma).\]
\end{theorem}

Note this construction is related to the perturbation $0\leq \cdots \leq 0 \leq f$.

\begin{proof}
We consider two cases: whether $\Sigma$ is a sphere or not.

\vspace{5pt}
{\bf $\Sigma$ is not a sphere.} If $\Sigma$ is not a sphere, then by \cite[Section 6]{colding2012generic}, there exists a positive smooth function $f$ defined on $\Sigma$ which is an entropy unstable variation, i.e. $\lambda(\Sigma_{\epsilon f})<\lambda(\Sigma)$ for $\epsilon$ very small. Thus we have
\[\lambda((m-1)\Sigma+\Sigma_{\epsilon f})\leq \lambda((m-1)\Sigma)+\lambda(\Sigma_{\epsilon f})<\lambda((m-1)\Sigma)+\lambda(\Sigma)=\lambda(m\Sigma).\]

\vspace{5pt}
{\bf $\Sigma$ is a sphere.} If $\Sigma$ is a sphere, we choose $f\equiv 1$. Then $\Sigma_{\epsilon f}$ is also a sphere, but has different radius. Then the entropy of $\Sigma_{\epsilon f}$ is not achieved by $F_{0,1}(\Sigma_{\epsilon f})$. However the entropy of $\Sigma$ is achieved by $F_{0,1}(\Sigma)$. Thus 
\[\lambda((m-1)\Sigma+\Sigma_{\epsilon f})<\lambda((m-1)\Sigma)+\lambda(\Sigma_{\epsilon f})=\lambda((m-1)\Sigma)+\lambda(\Sigma)=\lambda(m\Sigma).\]
Here we use Proposition \ref{Prop:entropy sum and sum of entropy}.
\end{proof}

\section{Ilmanen's Analysis of Mean Curvature Flow Near Singularity}\label{S:Ilmanen's Analysis of MCF Near Singularity}
In the following three sections, we will be concentrating on surfaces in $\mb R^3$.

\vspace{5pt}
In \cite{ilmanen1995singularities}, Ilmanen proved that a blow up limit of the mean curvature flow of closed embedded surfaces at the first singular time is a smooth embedded self-shrinker. In this section we review a key step in Ilmanen's proof (\cite[Proof of Theorem 1.1]{ilmanen1995singularities}). 

Suppose $\{M_t\}_{t\in[0,T)}$ is a family of smooth embedded surfaces flowing by mean curvature. The key step shows that near the singular time $T$, we have a sequence of timing $t_i\to T$ such that the blowing up sequence $M_{t_i}^{\alpha^j}$ have uniform curvature integral bound besides a discrete set of points. As a result, these slices are multi-sheeted layers over the limit self-shrinker. For detailed proof, see Ilmanen's preprint \cite{ilmanen1995singularities}. Also see \cite{sun2018singularities} for the proof of mean curvature flow with additional forces. 

We sketch the proof here.

\begin{proof}[Sketch of Ilmanen' proof]
Let $\{M_t\}_{t\in[0,T)}$ be a family of surfaces in $\mb R^3$ flowing by mean curvature. Suppose $\{\alpha_j\}$ is a sequence of numbers such that $\alpha^j\to\infty$ as $j\to \infty$. By the local Gauss-Bonnet Estimate \cite[Theorem 3]{ilmanen1995singularities} and the estimate of mean curvature integral bound \cite[Lemma 6]{ilmanen1995singularities}, Ilmanen proved that (See \cite[Page 21-22]{ilmanen1995singularities})
there exists a sequence $\tau_j\to 0$, $R_j\to \infty$ and $t_j\in[-1-\tau_j,-1]$ so that
\[\int_{M_{t_j}^{\alpha_j}\cap B_{R_j}}|A|^2\leq C{R_j}^2+C+\delta_R(j)/\tau_j,\]
and
\[\int_{M_{t_j}^{\alpha_j}}\frac{1}{4\pi}e^{-\frac{|x|^2}{4}}\left|H-\frac{\langle x,\mathbf n\rangle}{-2t_j}\right|^2\leq \delta(j)/\tau_j,\]
where $\delta_R(j)/\tau_j\to 0$ and $\delta(j)/\tau_j\to 0$. Here note that Ilmanen used a different sign in the definition of mean curvature, hence in \cite{ilmanen1995singularities} the sign of mean curvature is different from this paper.

Define $\sigma_j=|A|^2\mc H^2\lfloor M_{t_j}^{\alpha^j}$. Then $\sigma_j$ is a Radon measure, hence we can choose a subsequence converging to a Radon measure $\sigma$. By \cite[Lemma 8]{ilmanen1995singularities}, $M_{t_j}^{\alpha^j}$ converges to a self-shrinker $\nu$ as Radon measure. Denote the support of $\nu$ by $\Sigma$.

Now we prove that when $j$ is large enough, $M_{t_j}^{\alpha^j}$ is a union of layers over the limited self-shrinker besides a set of discrete points. Let $\epsilon_0>0$ which will be determined later. Since for each ball $B_R$, $\sigma_j(B_R)$ has uniform bound, so $\sigma(B_R)$ is finite as well. Thus there are only finitely many points (hence discrete) $p_1,\cdots,p_k$ in $B_R$ such that $\sigma(p_i)>\epsilon_0$. 

Suppose $\epsilon\in (0,\epsilon_0]$, $p\in \Sigma$ and $r,R>0$ such that 
\[B_r(p)\subset B_R,\ \ \ r\in(0,\epsilon),\ \ \ \sigma(B_r(p))<\epsilon^2.\]
Let us write $M_i$ to be $M_{t_{j_i}}^{\alpha^{j_i}}$ where $\{j_i\}$ is a subsequence of $\{j\}$, and $\mu_i=\mc H^2\lfloor M_i$. When $i$ large enough, $\sigma_{j_i}(B_r(p))$ is also smaller than $\epsilon^2$, hence we can apply Simon's graph decomposition theorem (\cite[Theorem 10]{ilmanen1995singularities}, \cite[Lemma 2.1]{Simon-Willmore}) and its corollary (\cite[Corollary 11]{ilmanen1995singularities}) to prove that there exists $s_i\in[r/4,r/2]$ and decomposition of $M_l\cap B_{s_i(p)}$ into distinct components $D_{i,l}$'s. Each $D_{i,l}$ is an embedded disk with certain area bound. Then Allard's compactness theorem (\cite[6.4]{allard1972first}, \cite[Theorem 42.7, p.247]{simon1983lecture}) implies that there is a subsequence of $s_i\to s\in[r/4,r/2]$. with $\mc H^2\lfloor D_{i,l}\to \nu_l$ for some varifold limit $\nu_l$, $l=1,\cdots,m$. 

By the monotonicity formula, $\nu_l$ weakly solves the self-shrinker's equation, and as a result $|H|\leq \epsilon/r$. Then Allard's regularity theorem (\cite[Theorem 9]{ilmanen1995singularities}, \cite[Section 8]{allard1972first}) together with the area bound obtained from Simon's graph decomposition theorem imply that the support of $\nu_l$ is a graph of a $C^{1,\alpha}$ function over a domain in a $2$-plane. By Schauder estimate $\nu_l$ is actually a graph of a $C^\infty$ function over a smaller domain in a $2$-plane. Thus $\Sigma$ is an immersed manifold in a small ball $B_{s/2}(p)$.

As a result, we proved that $\Sigma$ is a smooth manifold besides a discrete set of points $\{p_1,p_2,\cdots\}$. By maximum principle, $\Sigma$ is actually smoothly embedded away from $\{p_1,p_2,\cdots\}$. Also note the self-shrinker's equation is the minimal surface equation in $\R^3$ with Gaussian metric, the standard removable of singularity theorem ( \cite{gulliver}, \cite[Proposition 1]{choi1985space}) implies that $\Sigma$ is actually a smooth embedded everywhere. 

Moreover, Ilmanen proved that away from the small neighbourboods of those curvature concentrating points, $M_i$ is decomposed to several layers, See \cite[p.24-27]{ilmanen1995singularities}. More precisely, there exist $r_i\to 0$ such that $M_i\backslash (\cup_{i=1}^k B_{r_i}(p_k))$ is the union of $m$ connected components. Each component is called a layer, which is homotopic to $\Sigma\backslash\{p_1,\cdots,p_k\}$. As $i\to \infty$, each layer converges to $\Sigma$ as Radon measure.
\end{proof}

\begin{remark}
Ilmanen's proof relies on some arguments which only hold for surfaces in $\R^3$. For example the generalized Gauss-Bonnet theorem and Simon's graphical decomposition theorem. The smoothness also relies on the fact of the maximum principle and embeddedness of the flow. Hence the proof only works when the family of surfaces are embedded surfaces of codimension one. Thus the proof of smoothness and embeddedness of the tangent flow only works for surfaces in $\R^3$ flowing by mean curvature.

In \cite[p.28]{ilmanen1995singularities}, Ilmanen also gave an example to show that in higher codimensions the tangent flow may not be smooth. So far it is not known whether the tangent flow of a family of closed smooth embedded hypersurfaces in $R^{n+1}$ flowing by mean curvature is smooth for $n>2$.
\end{remark}

\begin{remark}
The uniqueness is another issue that we do not completely understand. 

So far there are some partial results on the uniqueness of the tangent flow. For example, Colding-Ilmanen-Minicozzi \cite{colding-ilmanen-minicozzi} proved that if one of the tangent flow at a singular point is a multiplicity one cylinder, then all the tangent flows at that singular point is a multiplicity one cylinder. Later Colding-Minicozzi \cite{colding-minicozzi_lojasiewicz} proved that those cylinders are in a fixed direction. If the tangent flow is a closed shrinker with multiplicity one, Schulze \cite{schulze_uniqueness} proved the uniqueness of the tangent flow. If the mean curvature flow is a family of smooth embedded surfaces in $\R^3$, Bernstein-Wang \cite{bernstein-wang_uniqueness} proved that even with higher multiplicity, if the tangent flow is a sphere, a cylinder or a plane, then it must be unique up to $SO(3)$ action on $\R^3$.

It would be interesting to know whether the tangent flow is unique for more general cases.
\end{remark}

To conclude this section, we use Ilmanen's analysis to prove that those blow up slices $M_i=M_{t_{j_i}}^{\alpha^{j_i}}$ are actually supported on a small tubular neighbourhood of $\Sigma$ if the $\Sigma$ is a closed self-shrinker.

\begin{corollary}\label{Cor:M_i is supported on bounded ball}
Suppose the blow up limit at $T$ is a closed self-shrinker $\Sigma$. For given $\delta>0$ such that the $\delta$-tubular neighbourbood $N_\delta\Sigma$ is well-defined, there exists $\tilde i>0$ such that $M_i$ is supported on the $N_\delta\Sigma$ for $i>\tilde i$.
\end{corollary}

\begin{proof}
We follow the notations in the previous proof. From the proof we know that besides a finite set of points $\{p_1,\cdots,p_k\}$, for any $p\in N_\delta\Sigma\backslash N_{\delta/2}\Sigma$, we have $\sigma(p)\leq \epsilon_0$. We claim that there is $r_p>0$ such that $B_{r_p}(p)\cap M_i$ is empty for $i$ sufficiently large.

We argue by contradiction. If this is not true, then we can conduct the analysis as in the proof to prove that for a subsequence of $M_i$'s intersecting with $B_{r}(p)\subset N_\delta\Sigma\backslash N_{\delta/2}\Sigma$, there is $s_i\in[r/4,r/2]$ such that $M_i\cap B_{s_i}(p)$ can be decomposed into distinct embedded disks with area bound, then by Allard's compactness theorem a subsequence of $M_i\cap B_{s_i}(p)$ converge to a limit, which should be a part of $\Sigma$. This limit is nontrivial because we have area lower bound (see \cite[p.18 Corollary 11]{ilmanen1995singularities}). This is a contradiction to the fact that $\Sigma$ is supported on $N_{\delta/2}\Sigma$.

This argument works for all $p\in N_\delta\Sigma\backslash N_{\delta/2}\Sigma$. Now let us choose $\delta_1\in(\delta/2,\delta)$ such that $\partial N_{\delta_1}\Sigma$ does not intersect with any one of $\{p_1,\cdots,p_k\}$. This can be done because $\{p_1,\cdots,p_k\}$ is finite. Then for every $p\in \partial B_{R_1}$ we can run the above argument to find a ball $B_r(p)$ such that for $i$ large enough, $M_i\cap B_r(p)$ is empty. So these $B_r(p)$'s form an open cover of $\partial N_{\delta_1}\Sigma$, and we can find a finite cover of $\partial N_{\delta_1}\Sigma$ by these $B_r(p)$'s. So there is $i_0>0$ such that for all $i>i_0$, $M_i\cap \partial N_{\delta_1}\Sigma$ is empty. Since $M_i$ is connected, $M_i$ is supported on $N_{\delta_1}\Sigma\subset N_{\delta}\Sigma$. Then we conclude the corollary.

\end{proof}

This corollary will be used to prove that many continuity arguments actually can be applied to the blow up sequence, and the flow will converge to a single point if the tangent flow is closed.

\section{Perturbation Near Singularities with Multiplicity}\label{S:Perturbation Near Singularities with Multiplicity}
In Section \ref{S:Entropy Instablity}, we have shown that for a closed self-shrinker $\Sigma$, there is a perturbation function $f$ over $\Sigma$ such that 
\[\lambda((m-1)\Sigma+(\Sigma_f))<\lambda(m\Sigma).\]
In Section \ref{S:Ilmanen's Analysis of MCF Near Singularity}, we have shown that near the singularity of higher multiplicity, the time slices of a blow up sequence are $m$-sheeted graphical, besides a finite collection of singular points. This fact suggests that we need to construct perturbations that are only defined away from these singularities. This is the first theorem of this section.

\begin{theorem}\label{Thm:variations away from finite points}
Given a closed self-shrinker $\Sigma$ and a smooth function $f>0$ defined on $\Sigma$. We require the $C^{2,\alpha}$ norm of $f$ is small enough such that $\Sigma_f$ is also an embedded hypersurface. Given finitely many points $p_1,\cdots,p_k\in\Sigma$ and a sequence of positive numbers $\{r_i\}_{i=1}^\infty$ with $r_i\to 0$ as $i\to \infty$, there exist a family of functions $\{f_i\}_{i=1}^\infty$ satisfying the following properties:
\begin{enumerate}
\item $f_i\geq 0$,
\item $f_i$ is supported on $\Sigma\setminus (\cup_{j=1}^k B_{r_i}(p_j))$;
\item $f_i=f$ on $\Sigma\setminus (\cup_{j=1}^k B_{2r_i}(p_j))$;
\item $\Sigma_{f_i}\to \Sigma_{f}$ as Radon measure.
\end{enumerate}
\end{theorem}

\begin{proof}
Let us fix $r_0>0$ such that for any $p_j$, $B_{r_0}(p_j)$ is a geodesic ball on $\Sigma$ such that the distance function $d_j(x)$, which is defined to be the distance from $x$ to $p_j$, is well defined. Moreover we assume that $B_{r_0}(p_j)$ does not intersect with each other for different $j$. Let us pick an arbitrary sequence of positive number $\{r_i\}_{i=1}^\infty$ such that $r_i<r_0/2$. We define a cut-off function $\varphi$ on $\mathbb{R}$ such that $0\leq \varphi\leq 1$, $\varphi(x)=1$ when $|x|\leq 1$, and $\varphi(x)=0$ when $|x|\geq 2$. Then we define the cut-off functions $\psi^i_j$ on $\Sigma$ to be
\[\psi^i_j(x)=\varphi(\frac{d_j(x)}{r_i}).\]
Then we define another function 
\[g_i=1-\sum_{j=1}^k\psi^i_j.\]
$g$ is a non-negative smooth function which is supported on $\Sigma\setminus (\cup_{j=1}^k B_{r_i}(p_j))$ and $g=1$ on $\Sigma\setminus (\cup_{j=1}^k B_{2r_i}(p_j))$. Finally, we define $f_i=g_if$.
It is straightforward to check that the properties (1),(2),(3) hold. 

Now we check that the property (4) holds. Note $\Sigma_{f_i}$ and $\Sigma_{f}$ are not identical only as the graphs over the balls $\cup_{j=1}^k B_{2r_i}(p_j)$. Let $\mc B^i_j$ be the normal neighbourhood of $B_{2r_i}(p_j)$ in $\mb R^3$ which contains the graph of $f_i$ and $f$ over $B_{2r_i}(p_j)$. Suppose $\mu_i$ is the Radon measure of $\Sigma_{f_i}$ and $\mu$ is the Radon measure of $\Sigma_f$, we only need to prove that $|\mu_i(\mc B^i_j)|+|\mu(\mc B^i_j)|\to 0$ as $i\to \infty$ to conclude that the Radon measures convergence.

$|\mu(\mc B^i_j)|\to 0$ is a consequence of dominant convergence theorem. Let us check that $|\mu_i(\mc B^i_j)|\to 0$. Note $|\nabla g_if|\leq f|\nabla g_i|+g_i|\nabla f|$, so if we write $B_{2r_i}(p_j)$ into local polar coordinate we can see that the area measure of $\Sigma_{f_i}$ is bounded by $C/r_i$ for some universal constant $C$. Thus the area of $\Sigma_{f_i}\cap \mc B^i_{j}$ is bounded by $Cr_i^2/r_i=Cr_i$, which converge to $0$ as $i\to\infty$. Thus we conclude that the property (4) holds.
\end{proof}

\begin{remark}
In the argument of property (4), the dimension $2$ is essential. We need the dimension $2$ to conclude that $|\mu_i(\mc B^i_j)|\to 0$. In general, the argument of property (4) holds for all dimensions greater than $1$, but it fails when $n=1$.
\end{remark}

Now we want to apply these perturbations to the blow up sequence. Before then we need the following lemma. This lemma transfers the perturbation over a surface into an isotopy of the tubular neighbourhood of the surface.

\begin{lemma}\label{L:construction of homotopy}
Let $\Sigma$ be a closed embedded surface and $f$ be a non-negative smooth function on $\Sigma$ such that $\Sigma_f$ is also a closed embedded surface. Suppose $N\Sigma$ is a tubular neighbourhood of $\Sigma$ such that $\Sigma_f\subset N\Sigma$. Then there exists an isotopy $h(\cdot,t)\colon N\Sigma\times[0,1]\to N\Sigma$ such that $h(\Sigma,0)=\Sigma$, $h(\Sigma,\epsilon)=\Sigma_{\epsilon f}$. 
\end{lemma}

\begin{proof}
We use $(x,s)$ to denote the point $x+s\mathbf{n}$ for $x\in\Sigma$. Let $\beta_1<\beta_2$ be positive numbers such that
\[\Sigma_f\subset\{(x,s)\colon x\in\Sigma,s\in[-\beta_1,\beta_1]\}\subset \{(x,s)\colon x\in\Sigma,s\in[-\beta_2,\beta_2]\}\subset N\Sigma.\]
Define $\varphi$ to be a smooth cut-off function on $\R$, which is supported on $[-\beta_2,\beta_2]$ and equals $1$ at $0$. Moreover we require $|D\varphi|< 1/\beta_1$. We define a map
\[h((x,s),t)=(x,s+t\varphi(s)f).\]

In order to show this is an isotopy, we need to check that for $t\in[0,1]$, $h(\cdot,t)$ is an injection. We argue by contradiction. Suppose $h((x,s_1),t)=h((x,s_2),t)$ for $s_1\neq s_2$. This means that
\[s_1+t\varphi(s_1)f(x)=s_2+t\varphi(s_2)f(x).\]
If $f(x)=0$, this is impossible. If $f(x)>0$, we have
\[\frac{\varphi(s_1)-\varphi(s_2)}{s_1-s_2}=-\frac{1}{tf(x)},\]
which contradicts the derivative bound of $\varphi$. So $h(\cdot,t)$ is the desired isotopy.
\end{proof}

Suppose $\{M_t\}_{t\in[0,T)}$ is a family of smooth closed embedded surfaces in $\mb R^3$ flowing by mean curvature. Moreover suppose $M_t$ has a singularity with multiplicity $m$ at time $T$, i.e. a blow up sequence of $M_t$ converge to a smooth closed embedded self-shrinker $\Sigma$ with multiplicity $m$ as Radon measure.

As we discussed in Section \ref{S:Ilmanen's Analysis of MCF Near Singularity}, Ilmanen proved that a blow up sequence which are multi-sheeted layers over $\Sigma$ beside a small neighbourhood of finitely many points. Let us review all the ingredients in Ilmanen's result:

\begin{enumerate}
\item A sequence $M_i=M_{t_i}^{\alpha_i}=\alpha_i(M_{T+{\alpha_i^{-2}t_i}})$, where $t_i\to -1$, $\alpha_i\to \infty$ as $i\to \infty$;
\item A finite collection of points $\{p_1,\cdots,p_k\}$;
\item A sequence of positive numbers $r_i\to 0$ as $i\to\infty$.
\end{enumerate}

Then Ilmanen proved that $M_i$'s are $m$-sheeted layers over $\Sigma$ outside the balls $\cup_{j=1}^k B_{r_i}(p_j)$, and the layers will converge to $\Sigma$ as $i\to \infty$.

Now we are going to prove that we can perturb $M_i$ such that the local entropy of the perturbed surface is strictly less than $\lambda(m\Sigma)$.

\begin{theorem}\label{Thm:perturbation near singularity with multiplicity}
For any $\epsilon_0>0$, $0<a<1$, there exists a positive integer $i$ and a perturbed surface $\tilde M_i$ of $M_i$, such that $\tilde M_i$ is isotopic to $M_i$ and the Hausdorff distance between $M_i$ and $\tilde M_i$ is smaller than $\epsilon_0$, such that $\lambda^{[a,\infty)}(\tilde M_i)<\lambda(m\Sigma)$.
\end{theorem}

\begin{proof}
By the construction of Theorem \ref{Thm: Construction of unstable perturbation of self-shrinker with multiplicity}, we can choose $\epsilon_1<\epsilon_0$ and a positive function $f$ such that for $0<\epsilon<\epsilon_1$, 
\[\lambda((m-1)\Sigma+\Sigma_{\epsilon f})<\lambda(m\Sigma).\]
Then Theorem \ref{Thm:variations away from finite points} implies that we can find a sequence of non-negative functions $\{f_l\}_{l=1}^\infty$, $f_l$ is supported on $\Sigma\backslash(\cup_{j=1}^k B_{r_l}(p_j))$, and $\Sigma_{\epsilon f_l}\to \Sigma_{\epsilon f}$ as Radon measure.

Recall that $M_i$ are $m$-sheeted layers over $\Sigma$ outside $\cup_{j=1}^k B_{r_i}(p_j)$. Given $\beta>0$, by Corollary \ref{Cor:M_i is supported on bounded ball}, $M_i$ is supported in a $\beta$-tubular neighbourhood of $\Sigma$. Since $M_i$ is embedded, we can label these layers by their heights over $\Sigma$. Now let us define a perturbed surface $\tilde M_{il}$ for $l\leq i$, which coincides with $M_i$ but we modify the highest layer by applying the isotopy $h(\cdot,\epsilon_1 f_l)$ to it, where $h(\cdot,t)$ is the isotopy defined in Lemma \ref{L:construction of homotopy}. More precisely, let $L$ be the highest layer of $M_i$, then we define $\tilde M_{il}$ to be
\[(\Sigma\backslash L)\cup h(L,\epsilon_1 f_l).\]
Since $h(\cdot,t)$ is an isotopy, $\tilde M_{il}$ is still a smooth closed embedded surface. We only need to show that $\tilde M_{il}$ has local entropy strictly less than $\lambda(m\Sigma)$ for $i>l$ large. Suppose 
\[\lambda((m-1)\Sigma+\Sigma_{\epsilon_1 f})=\lambda(m\Sigma)-\delta.\]
By (4) of Theorem \ref{Thm:variations away from finite points} and the continuity of local entropy,
\[\lambda^{[a,\infty)}((m-1)\Sigma+\Sigma_{\epsilon_1 f_l})\to \lambda^{[a,\infty)}((m-1)\Sigma+\Sigma_{\epsilon_1 f})=\lambda((m-1)\Sigma+\Sigma_{\epsilon_1 f}).\]
Hence when $l$ large enough, 
\[\lambda^{[a,\infty)}((m-1)\Sigma+\Sigma_{\epsilon_1 f_l})<\lambda(m\Sigma)-\delta/2.\]

Note that $M_{i}$ converge to $m\Sigma$ as Radon measure implies that $\tilde M_{il}$ converge to $((m-1)\Sigma+\Sigma_{\epsilon_1 f_l})$ as Radon measure as $i\to \infty$. Then by the continuity of local entropy, 
\[\lambda^{[a,\infty)}(\tilde M_{il})\to \lambda^{[a,\infty)}((m-1)\Sigma+\Sigma_{\epsilon_1 f_l})\]
as $i \to \infty$, hence when $i>l$ large enough,
\[\lambda^{[a,\infty)}(\tilde M_{il})<\lambda(m\Sigma)-\delta/3<\lambda(m\Sigma).\]
Then $\tilde M_i=\tilde M_{il}$ is a desired perturbation. Note that $\epsilon_1$ can be arbitrarily chosen, thus we can require the Hausdorff distance between $M_i$ and $\tilde M_i$ to be as small as we want.
\end{proof}

\begin{remark}
We hope the same perturbation holds for non-compact self-shrinker singularity as well. However some technical issues arise. 

The main issue is that most of the continuity arguments of local entropy fail. We do not even know that whether the entropy (or local entropy) is achieved by a point $(x_0,t_0)$ for a non-compact surface. For example, Guang \cite{guang2016volume} proved that grim reaper surface and bowl soliton, two translating solitons of mean curvature flow in $\mb R^3$, achieve their entropy at infinity. Even the entropy or local entropy of $\Sigma_i$ is achieved at a fixed point $(x_i,t_i)$, let $\Sigma_i\to \Sigma$, $|x_i|$ may go to infinity and the continuity fails.
\end{remark}

\section{Proof of Main Theorem \ref{Thm:Main Theorem}}\label{S:Proof of Main Theorem}
Let $\{M_t\}_{t\in[0,T)}$ be a family of closed embedded surfaces flowing by mean curvature. Suppose at time $T$ the mean curvature flow has a singularity at $y\in\mathbb{R}^3$, and assume that the blow up limit is a self-shrinker $\Sigma$ with multiplicity $m>1$. In the previous section, we proved that near the singularity we can perturb a time slice of the blow up sequence of $M_t$ at $(y,T)$ such that the local entropy of the surface after the perturbation is strictly less than $\lambda(m\Sigma)$. In this section, we are going to prove that $m\Sigma$ can never show up as a singularity after the perturbation.

The following lemma shows that after the perturbation we constructed in Theorem \ref{Thm:perturbation near singularity with multiplicity}, a closed shrinker will never show up as a singularity in a short time.

\begin{lemma}\label{L:no immediate closed singularity appear}
Let $\{M_t\}_{t\in[0,T)}$ be a family of hypersurfaces flowing by mean curvature, generating a closed singularity at time $T$. Let $\tilde M_i$ be the perturbed surface as in Theorem \ref{Thm:perturbation near singularity with multiplicity}. There exits $\delta>0$ depending on $f_l$, $i_0>0$ such that when $i\geq i_0$, the mean curvature flow $\tilde M_t$ starting from the perturbed surface $\tilde M_i$ at time $t_i$, does not have a first time singularity in the time interval $t\in[t_i,t_i+\delta]$ which is modelled by a closed self-shrinker.
\end{lemma}

The proof is based on a simple observation: if the mean curvature flow has a singularity which is modelled by a closed shrinker, then the flow itself must converge to a single point at the singular time.

\begin{proof}
The blow up limit of $M_t$ at $(y,T)$ being a closed self-shrinker implies that $M_t$ will shrink to a single point as $t\to T$. This is straightforward from Corollary \ref{Cor:M_i is supported on bounded ball}. In order to prove that the mean curvature flow starting from $\tilde M_{i}$ can not have a closed self-shrinker as a first time blow up limit in the time interval $[t_i,t_i+\delta]$, we only need to prove that $\tilde M_t$ cannot shrink to a single point when $t\in [t_i,t_i+\delta]$.

Consider a ball $B_\gamma(x)\subset \R^3$, which lies in the domain bounded by $\Sigma_{\epsilon_1 f_l}$ but does not intersect with the closure of the domain bounded by $\Sigma$. Here $\gamma$ only depends on the function $f_l$. Then there exists $i_0>0$ such that when $i>i_0$, $B_\gamma(x)$ lies in the domain bounded by $\tilde M_i$. Then by the parabolic maximum principle (\cite[Proposition 3.3]{ecker2012regularity}, \cite[Theorem 2.2.1]{mantegazza}), the mean curvature flow starting from $\tilde M_i$ cannot shrink to a single point before $\partial B_\gamma(x)$ shrinking to a single point. So if the mean curvature flow starting from $\partial B_\gamma(x)$ at time $0$ shrink to a single point at time $\delta$, then the mean curvature flow starting from $\tilde M_i$ at time $t_i$ cannot shrink to a single point for $t\in[t_i,t_i+\delta]$.
\end{proof}

Now we prove the main Theorem.

\begin{theorem}{(Main Theorem \ref{Thm:Main Theorem})}
Suppose $\{M_t\}_{t\in [0,T)}$ is a family of smooth closed embedded surfaces in $\mb R^{3}$ flowing by mean curvature, and suppose a tangent flow of $M_t$ at $(y,T)$ for some $y\in\R^3$ is $\{m\cdot \sqrt{-t}\Sigma\}_{t\in(-\infty,0)}$ with multiplicity $m>1$. Given $\epsilon_0>0$, there exists $t_0\in[0,T)$ such that there is a perturbed closed embedded surface $\tilde M_{t_0}$ which is isotopic to $M_{t_0}$, has Hausdorff distance to $M_{t_0}$ less than $\epsilon_0\sqrt{T-t_0}$, and if $\{\tilde M_t\}_{t\in[t_0,T')}$ is a family of closed embedded surfaces flowing by mean curvature starting from $\tilde M_{t_0}$, the tangent flow will never be $\{m\cdot \sqrt{-t}\Sigma\}_{t\in(-\infty,0)}$. 
\end{theorem}

\begin{proof}
By the construction in Theorem \ref{Thm:perturbation near singularity with multiplicity}, given $\epsilon_0>0$ and $a<\delta$ where $\delta$ comes from Lemma \ref{L:no immediate closed singularity appear}, we can always find an $M_i=M_{t_i}^{\alpha_i}$ and a perturbed surface $\tilde M_{i}$ of $M_i$ such that $\lambda^{[a,\infty)}(\tilde M_i)<\lambda(m\Sigma)$. We choose $t_0=T+t_i\alpha_i^{-2}$ and $\tilde M_{t_0}=\alpha_i^{-1} \tilde M_i+y$. We remind the readers that $M_i$ and $\tilde{M}_i$ are surfaces that we obtain after rescaling the original flow, while $\tilde M_{t_0}$ is the perturbation of $M_{t_0}$, that is an original flow time slice. We are going to prove that this is the desired perturbed surface. 

$\lambda^{[a,\infty)}(\tilde M_{i})<\lambda(m\Sigma)$ implies that 
\[\lambda^{[a\alpha_i^{-2},\infty)}(\tilde M_{t_0})<\lambda(m\Sigma),\]
hence by Theorem \ref{Thm:no certain type singularity in infinite interval}, the mean curvature flow $\tilde M_t$, starting from $\tilde M_{t_0}$, cannot generate a tangent flow $\{m\cdot \sqrt{-t}\Sigma\}_{t\in(-\infty,0)}$
in the time interval $[t_0+a\alpha_i^{-2},\infty)=[T+t_i\alpha_i^{-2}+a\alpha_i^{-2},\infty)$. 

Finally, by Lemma \ref{L:no immediate closed singularity appear}, $\tilde M_t$ can not have a blow up limit to be $\{m\cdot \sqrt{-t}\Sigma\}_{t\in(-\infty,0)}$ in the time interval $[T+t_i\alpha_i^{-2},T+(t_i+\delta)\alpha_i^{-2}]\supset[t_0,t_0+a\alpha_i^{-2}]$. Then we conclude that $\tilde M_t$ cannot have a first time blow up limit to be $\{m\cdot \sqrt{-t}\Sigma\}_{t\in(-\infty,0)}$ in the time interval $[t_0,\infty)$.
\end{proof}

\bibliography{bibfile-2}
\bibliographystyle{alpha}

\end{document}